\documentclass[a4paper,11pt,leqno]{article}
\usepackage{amsmath,amssymb,amsfonts,amsthm,graphicx,fancyhdr,epic}
\usepackage[latin1]{inputenc}
\usepackage[T1]{fontenc}
\usepackage{rotating,datetime,mathrsfs,hhline,multirow,float}

\newdateformat{mydate}{\THEYEAR/\THEMONTH/\THEDAY}

\newtheorem{teo}{Theorem}[section]
\newtheorem{lem}[teo]{Lemma}
\newtheorem{cor}[teo]{Corollary}
\newtheorem{defi}[teo]{Definition}
\newtheorem{ques}[teo]{Question}
\newtheorem{conj}[teo]{Conjecture}

\newtheoremstyle{drem}
     {3pt}
     {3pt}
     {\rmfamily}
     {}
     {\bf}
     {:}
     {.5em}
     {}

\theoremstyle{drem}
\newtheorem{rem}[teo]{Remark}

\newcommand{\vstr}[0]{{\vrule width 0in height 3ex depth 0in}}

\newcommand{\eg}[0]{\emph{e.g.} }
\newcommand{\ie}[0]{\emph{i.e.} }

\newcommand{\pgen}[1]{\langle #1 \rangle}

\newcommand{\ssl}[1]{\underline{#1}}
\newcommand{\srl}[1]{\overline{#1}}

\newcommand{\eps}[0]{\epsilon}

\newcommand{\rr}[0]{\ensuremath{\mathbb{R}}}
\newcommand{\zz}[0]{\ensuremath{\mathbb{Z}}}
\newcommand{\nn}[0]{\ensuremath{\mathbb{N}}}

\newcommand{\jo}[1]{\ensuremath{\mathcal{#1}}}

\DeclareSymbolFont{bbold}{U}{bbold}{m}{n}
\DeclareSymbolFontAlphabet{\mathbbold}{bbold}
\newcommand{\un}[0]{\mathbbold{1}}

\title{The Liouville property and Hilbertian compression}
\author{Antoine Gournay}
\date{\mydate\today}

\begin{document}

\maketitle

\begin{abstract}
Lower bound on the equivariant Hilbertian compression exponent $\alpha$ are obtained using random walks. More precisely, if the probability of return of the simple random walk is $\succeq \textrm{exp}(-n^\gamma)$ in a Cayley graph then $\alpha \geq (1-\gamma)/(1+\gamma)$. This motivates the study of further relations between return probability, speed, entropy and volume growth. For example, if $|B_n| \preceq e^{n^\nu}$ then the speed exponent is $\leq 1/(2-\nu)$. 

Under a strong assumption on the off-diagonal decay of the heat kernel, the lower bound on compression improves to $\alpha \geq 1-\gamma$. Using a result from Naor \& Peres \cite{NPspeed} on compression and the speed of random walks, this yields very promising bounds on speed and implies the Liouville property if $\gamma <1/2$.
\end{abstract}

\section{Introduction}

Throughout the text, $G$ will be a finitely generated discrete group and it will be studied using its Cayley graph. The finite symmetric generating set $S$ chosen to produce the Cayley graph will not be explicitly mentioned unless it is of importance; finiteness and symmetry will also always be assumed. $P$ is the distribution of a lazy random walk. More precisely, it is obtained from a simple random walk distribution $P' = \un_S/|S|$ by $P = \tfrac{1}{2} (\delta_e + P')$. $P^{(n)}$ is the $n^\text{th}$-step distribution of the lazy random walk, \ie the $n^\text{th}$-convolution of $P$ with itself.

For further definitions, the reader should consult \S{}\ref{sdefi}.
\begin{teo}\label{tlrwcomp-t}
If $P^{(n)}(e) \geq L e^{-K n^{\gamma}}$ where $L,K >0$
then the equivariant compression exponent of $G$, $\alpha(G)$, satisfies $\alpha(G) \geq (1- \gamma)/(1+\gamma)$. 
\end{teo}
This improves a lower bound from Tessera \cite[Proposition 15]{Tessera}: $\alpha(G) \geq (1- \gamma)/2$. The proof of Theorem \ref{tlrwcomp-t} is contained in \S{}\ref{slrw}. Recall that return probability are stable under quasi-isometries between Cayley graphs (see Pittet \& Saloff-Coste \cite[Theorem 1.2]{PSCstab}). There are many possible behaviours for $\gamma$, see Pittet \& Saloff-Coste \cite[Theorem 1.1]{PSCexp}. 

The speed [or drift] of a random walk is defined as $\mathbb{E}|P^{(n)}| =  \int |g| \mathsf{d} P^{(n)}(g)$ where $|g|$ is the word length of $g$ (\ie the graph distance in the Cayley graph between $g$ and the identity element). The speed [or drift] exponent is $\beta = \sup \{ c \in [0,1] \mid$ there exists $K >0$ such that $\mathbb{E}|P^{(n)}| \geq K n^c\}$. Surprisingly, there is little known on how much $\beta$ depends on $S$. 

Naor \& Peres showed in \cite[Theorem 1.1]{NPspeed} that $\alpha(G) \leq 1/2\beta$. Since the map $n \mapsto \mathbb{E}|P^{(n)}|$ is sub-additive, the sequence $\mathbb{E}|P^{(n)}|/n$ always as a limit. Compression is a natural way to show that $\beta$ is bounded away from $1$ (for any generating set) and hence, that the afore-mentioned limit is $0$. This is interesting because a group has the Liouville property if and only if $\mathbb{E}|P^{(n)}|$ is $\mathsf{o}(n)$.

However, the above result on compression only yields $\beta \leq (1+\gamma)/2(1-\gamma)$ which is non-trivial only if $\gamma < \tfrac{1}{3}$.
Let $B_n$ be the ball of radius $n$, \ie $B_n = \{ g \in G \mid |g| \leq n\}$. Recall that, if there are $K,L>0$ so that
\stepcounter{teo}
\[ \tag{\text{\theteo}} \label{eqvarop}
\forall n, |B_n| \geq K e^{Ln^v} \text{ then } \forall n, P^{(n)}(e) \leq K' e^{L n^c} \text{ with } c \leq \frac{v}{2+v}, 
\]
for some $K',L'>0$ (\eg see \cite[(14.5) Corollary]{Woess}). Hence the bound on speed is not interesting from the point of view of the Liouville property: by \eqref{eqvarop}, $\gamma < 1/3$ implies the group is of subexponential growth and so automatically Liouville. See \S{}\ref{sint} for more details.
However, this bound motivates further investigations on possible relations between the various quantities in groups of intermediate growth. 

Recall the entropy is defined by $H(P^{(n)}) := - \sum_{g \in G} P^{(n)}(g) \ln P^{(n)}(g)$.
\begin{teo}\label{tint-t} Assume $G$ is so that $|B_n| \leq Le^{Kn^{\nu}}$ for some $K,L>0$ and $\nu \in ]0,1]$.
\begin{enumerate}\setlength{\itemsep}{0ex} \renewcommand{\labelenumi}{ {\bf (\alph{enumi})}}
\item Then $\mathbb{E}|P^{(n)}| \leq K' n^{1/(2-\nu)}$ (hence $\beta \leq 1/(2-\nu)$)  and $H(P^{(n)}) \leq L'' + K'' n^{\nu/(2-\nu)}$ for some $K',K'',L''>0$.
\item $\alpha(G) \geq 1-\nu$.
\item If $H(P^{(n)}) \geq K' + L'n^h$ for some $K',L'>0$, then $\beta \nu \geq h$ and $ 2 h \alpha \leq \nu$.
\end{enumerate}
\end{teo}
For example, $\nu = \tfrac{k-1}{k}$ gives $\alpha \geq \tfrac{1}{k}$, $\mathbb{E}|P^{(n)}| \leq K' n^{k/(k+1)}$ and $H(P^{(n)}) \leq L''+K'' n^{(k-1)/(k+1)}$. 

The bound on speed extends to measures with finite second moment and improves the $\frac{1+\nu}{2}$ bound from Erschler \& Karlsson \cite[Corollary 13]{EK}. 
The upper bound on entropy also holds for measures of finite second moment, seems new and implies a result of Coulhon, Grigor'yan \& Pittet \cite[Equation (7.5) in Corollary 7.4]{CGP}.

The lower bound $\alpha(G) \geq 1-\nu$ is obtained as a corollary of Theorem \ref{tlrwcomp-t} and of the estimate on return probabilities coming from volume growth: $P^{(n)}(e) \geq K'' e^{L'' n^{\nu /(2-\nu)}}$. This lower bound is already present in Tessera \cite[Proposition 14]{Tessera} but comes here from a different method.

For more discussions on the various exponents in groups, see \S{}\ref{ssexp} and \S{}\ref{sint}.

The methods in the proof of Theorem \ref{tlrwcomp-t} give a particularly interesting result if one makes a strong hypothesis on the off-diagonal behaviour of the heat kernel. The most natural estimate which is conjectural but nevertheless relevant for the present purposes is the following: for some $M,N>0$
\[ \tag{\text{OD}} \label{eqodd}
P^{(n)}(g) \leq P^{(n)}(e) N e^{ -M |g|^2 /n}. 
\]
where $|g|$ is the word length of $g$ (\ie the distance between $g$ and the identity in the Cayley graph).
This estimate is true for groups polynomial growth (and free groups) but there are no other groups where it is known to hold. Weaker forms are sufficient, see \S{}\ref{ssoffdec} for details. Very recently, Brieussel \& Zheng \cite[Problem 9.3 and foregoing paragraphs]{BZ} have given an example of groups where this estimate fails, see also \S{}\ref{ssoffdec} below.

Before stating the next result recall that $P^{(n)}(e) \geq Ke^{-L n^\gamma}$ for $\gamma <1$ implies the group is amenable (see Kesten \cite{Kes2}). Furthermore, a result known as ``Gromov's trick'' shows non-equivariant compression is equal to equivariant compression in amenable groups.
\begin{teo}\label{tconj-t}
Assume that $P^{(n)}(e) \geq Ke^{-L n^\gamma}$ (with $\gamma <1$) in some Cayley graph of $G$ and \eqref{eqodd} holds in some [possibly different] Cayley graph of $G$. Then $\alpha(G') \geq 1- \gamma$ for any group $G'$ with a Cayley graph quasi-isometric to that of $G$. Consequently, $\beta \leq \frac{1}{2 (1-\gamma)}$ so that, if $\gamma< \tfrac{1}{2}$, the graph is Liouville.
\end{teo}
The bound obtained above is significantly more interesting; for example, if $\gamma = \tfrac{1}{3}$ it would yield $\beta \leq \tfrac{3}{4}$. Also, if $|B_n| \leq Ke^{Ln^\nu}$ it would yield, $\alpha(G) \geq \frac{1-\nu}{1-\nu/2}$. However the upper bound on the speed in Theorem \ref{tint-t} does not follow from Theorem \ref{tconj-t} if the estimate on the probability of return is only given by volume growth.

Theorem \ref{tconj-t}, the discussion below and \S{}\ref{stab} motivates the author to make the following
\begin{conj}\label{laconj}
If there are $K,L >0$ so that $P^{(n)}(e) \geq K e^{L n^\gamma}$ in a Cayley graph of $G$ then $\beta \leq 1/2(1-\gamma)$ in all Cayley graphs of $G$.
\end{conj}
This is now a theorem of Saloff-Coste \& Zheng \cite[Theorem 1.8]{SCZ} (their result is more precise than just an estimate on $\beta$ and covers many measure $P$ driving the random walk).

\textbf{Sharpness of Theorems \ref{tlrwcomp-t} and \ref{tconj-t}:} Nothing indicates Theorem \ref{tlrwcomp-t} is sharp. Sharpness of Theorem \ref{tconj-t} (assuming the hypothesis is satisfied!) are discussed in detail in \S{}\ref{stab}. In short, there are groups with $\gamma = 0, \tfrac{1}{3}, \tfrac{1}{2}$ or $1$ for which, if \eqref{eqodd} were to hold, Theorem \ref{tconj-t} is sharp (\ie $\alpha = 1-\gamma$; also $\beta = 1/2(1-\gamma)$ if $\gamma \neq 1$). There are also groups with $\gamma = \tfrac{1}{3}, \tfrac{1}{2}$ or $1$ where the conjectural bounds of Theorem \ref{tconj-t} meet neither compression nor speed. Thus, it seems unlikely that there is a better estimate in terms of those quantities (see Question \ref{laques} for a possible improvement).

Bartholdi \& Erschler \cite[\S{}1.2 and \S{}7]{BE} showed that some groups of intermediate growth have arbitrarily bad compression, in particular $\alpha = 0$. Consequently, there are Liouville groups with arbitrarily quickly decaying return probability (hence return exponent $\gamma >1/2$). Also, since growth is an invariant of quasi-isometry, the stability under quasi-isometry of the Liouville property is known for this class of groups. 

On the other hand, recent work of M.~Kotowski \& Vir\'ag \cite{KV} show there are groups with $- \ln P^{(n)}(e) \lesssim n^{1/2} + o(1)$ (the ``error'' being at most $\ln \ln n / \ln n$) which are not Liouville. 

An interesting point of investigation would be to determine whether all groups with $P^{(n)}(e) \asymp e^{-n^{1/2}}$ are Liouville (or exhibit a counterexample).

\textbf{Around Theorem \ref{tint-t}:} It is difficult to discuss the sharpness of Theorem \ref{tint-t} because the present construction of groups intermediate growth focus on controlling one parameter. These constructions often leave, in the meantime, the other parameters uncomputed (and hard to compute). It might, for this precise reason be even more interesting to have bounds between those quantities (see Amir \cite{Amir} or  Brieussel \& Zheng \cite{BZ} for recent developments). In fact, too good improvements of the bounds in Theorem \ref{tint-t} would lead to some forms of the gap conjecture on volume growth. This leads the author to believe that these are sharp.

\textbf{Isoperimetry:} How slowly must the F{\o}lner function of a group grow so that one can deduce that the group is Liouville? Theorem \ref{tconj-t} hints at an answer using the link between the F{\o}lner function and return probability from Bendikov, Pittet \& Sauer \cite{BPS}.

There are also descriptions in term of ``adapted isoperimetry''. For ``F{\o}lner couples'' the reader is referred to Coulhon, Grigor'yan \& Pittet \cite[Theorem 4.8]{CGP}). For ``controlled F{\o}lner sequences'' (and its relation to compression) see Tessera \cite[Corollary 13]{Tessera}. 
Of course, ``adapted isoperimetry'' mixes distances and isoperimetry, and are \emph{a priori} not completely determined by the F{\o}lner function. 

\textbf{Amenability:} 
The method presented in section \S{}\ref{slrw} is reminiscent of Bekka, Ch\'erix \& Valette \cite{BCV}. To show amenable groups have the Haagerup property, they used $w_n = \un_{F_n}$ where $F_n$ is a F{\o}lner sequence. See also Valette \cite[Proposition 1 in \S{}2]{Val}.

Recently, M.~Carette \cite{Car} showed that the Haagerup property is not an invariant of quasi-isometry; in \cite[Appendix A]{Car}, Arnt, Pillon \& Valette use these same examples to show that the equivariant compression exponent is not an invariant of quasi-isometry.

\textbf{Compression of Thompson's group $F$:} It is straightforward to reread the paper of Naor \& Peres \cite{NPspeed} [and/or the current text] while keeping track of compression functions instead of taking only the exponent. 
Introduce $s^{-1}(k) = \inf \{k \in \rr \mid \mathbb{E}|P^{(n)}| < k \}$. Under the (mild) assumption that $\rho_-$ is concave, 
then, $\rho_-$ is less (up to constants) than $k \mapsto (s^{-1}(k))^{1/2}$. Hence, a compression function strictly better than $n \mapsto K n^{1/2}$ implies the Liouville property. As noted in \cite{NPspeed} this improves a result of Guentner \& Kaminker \cite{GK} (since the Liouville property implies amenability). 

Here is an application of this remark. It seems known (see Kaimanovich \cite{Kai}) that Thompson's group $F$ is not Liouville (this does not have any impact on its amenability). In the case of non-Liouville groups the concavity hypothesis may be discarded (by using arguments from Austin, Naor \& Peres \cite{ANP}). This provides the answer to a question of Arzhantseva, Guba \& Sapir \cite[Question 1.4]{AGS}: the best Hilbertian equivariant compression function for Thompson's group $F$ is (up to constants) $\rho_-(x) \simeq x^{1/2}$. 

\emph{Acknowledgements:} The author would like to express his thanks to N.~Matte Bon for the examples of non-sharpness, to C.~Pittet for the references on the possible behaviours of the return probability and its invariance under quasi-isometry, to A.~Valette for the discussion around amenability, to N.~Matte Bon and A.~Erschler for the discussion on adapted isoperimetry, to T.~Pillon for discussions on the compression of Thompson's group $F$ and to Y.~Peres for pointing out an important correction in a previous version of this paper. This paper came out of discussions during the Ventotene 2013 conference and the author gratefully acknowledges the support of its sponsors.

\section{Definitions and preliminary results}\label{sdefi}

\renewcommand{\thesubsection}{\thesection.\Alph{subsection}}

Cayley graphs are defined by right-multiplication: $x$ and $y$ are neighbours if $\exists s \in S$ such that $xs = y$. Though common for the setting of random walks, this convention is slightly uncommon when one speaks of actions and convolutions.

The word length (for the implicit generating set $S$) of an element $g$ will be noted $|g|$.

\subsection{Compression}
\begin{defi}
Let $B$ be a Banach space and $\pi: G \to \mathrm{Isom}B$ be a representation of $G$ in the isometries of $B$. An equivariant uniform embedding $f: \Gamma \to B$ is a map such that there exist an unbounded increasing function $\rho_-: \rr_{\geq 0} \to \rr_{\geq 0}$ and a constant $C>0$, satisfying $\forall x,y \in \Gamma$
\[
\rho_-(|y^{-1}x|) \leq \|f(x) - f(y)\| \leq C |y^{-1}x| + C,
\]
and $f(\gamma x) = \pi(\gamma) f(x)$.\\
The function $\rho_-:\rr_{>0} \to \rr_{>0}$ is called the compression function (associated to $f$). 
The [equivariant] compression exponent is $\alpha(f) = \sup \{ c \in [0,1] \mid \exists K>0$ such that $\rho_-(n) \geq K n^c \}$.\\
The compression exponent of $G$, $\alpha(G)$, is the supremum over all $\alpha(f)$.
\end{defi}
It follows easily from the definition that changing the generating set does not change $\alpha$.

An equivariant uniform embedding is, in fact, very constrained. Indeed, one may (by translating everything) always put $f(e) = 0 \in B$ for simplicity. Next, recall that an isometry of a Banach space is always affine (Mazur-Ulam theorem). Write $\pi(y) v = \lambda(y) v + b(y)$ where $\lambda$ is a map from $G$ into the linear isometries of $B$ and $b$ is a map from $G$ to $B$. Note that $f(y) = \pi(y) f(e) = \pi(y) 0 =  b(y)$. Furthermore $\pi(xy) v= \pi(x) \pi(y)v$ (for all $v \in B$) implies that $\lambda$ is a homomorphism and $b$ satisfies the cocycle relation:
\[
b(xy) = \lambda(x) b(y) + b(x).
\]
The strategy that will be used here to make an interesting equivariant uniform embedding (\ie a $\lambda$-cocycle) is to use a ``virtual coboundary''. A coboundary would be a cocycle defined by
\[
f(y) = \lambda(y) v - v
\]
for some $v \in B$. The idea is to define such a cocycle using a $v$ which does belongs to $B$ but to some bigger space $\widetilde{B}$ (to which the action $\lambda$ extends). Note that if $f(s)$ belongs to $B$ for any $s$ in the generating set $S$, then this also holds for $f(g)$ for any $g \in G$ (thanks to the cocycle relation).

Finally, a quick calculation (using that $\lambda$ is isometric and writing $g$ as a word) shows that cocycles always satisfy the upper bound required by equivariant uniform embedding. Also, it suffices to check that $\|f(g)\| \geq \rho_-(|g|)$:
\[
\| b(gh) - b(g) \| = \| \lambda(g)b(h)\| = \|b(h)\|
\]
This explains why \S{}\ref{slrw} only discusses this lower bound.

\subsection{Probabilistic parameters for groups}\label{ssexp}

The entropy of a probability measure $Q$ is $H(Q) = -\sum_{g \in G} Q(g) \ln Q(g)$ (when convergent). The group $G$ is Liouville (for the [finite symmetric] generating set $S$) if any of the following equivalent conditions hold:
\newcounter{loulou} \setcounter{loulou}{0}
\newcommand{\liou}{\indent \stepcounter{loulou}(\roman{loulou})}
\begin{center}
\begin{tabular}{cl}
\liou & There are no non-constant bounded harmonic functions on the Cayley graph;\\
\liou & $H(P^{(n)})$ is $\mathsf{o}(n)$;\\
\liou & $\mathbb{E}|P^{(n)}|$ is $\mathsf{o}(n)$.\\
\end{tabular}
\end{center}
(iii)$\implies$(ii) can be obtained as in Lemma \ref{tvitinf-l}; see also Erschler \cite[Lemma 6]{Ers}. The implication (ii)$\implies$(i) may be found in Avez \cite{Avez72}. For a complete (and more modern) picture see Erschler \& Karlsson \cite{EK} and references therein. 

Recall that $\mathbb{E}|P^{(n+m)}| \leq \mathbb{E}|P^{(n)}| + \mathbb{E}|P^{(m)}|$ and $H(P^{(n+m)}) \leq H(P^{(n)}) + H(P^{(m)})$. Let $f:\rr_{\geq 0} \to \rr_{\geq 0}$ be increasing, unbounded, $f(n+m) \leq f(n) + f(m)$ and $f(0) =0$. Recall that $\lim_{n \to \infty} f(n)/n$ exists. One can also define two exponents:
\[
\begin{array}{rl}
\srl{\phi}	&= \inf \{ c \in [0,1] \mid \exists K>0,L\in \rr \text{ such that } f(n) \leq L + Kn^c \text{ for all } n \}			\\
		&= \sup \{ c \in [0,1] \mid \exists K>0,L\in \rr \text{ such that } f(n) \geq L + Kn^c \text{ for infinitely many } n \}		\\
		&= \displaystyle\vstr \limsup_{n \to \infty}  \frac{\ln f(n)}{\ln n} \\
\ssl{\phi}	&= \sup \{ c \in [0,1] \mid \exists K>0,L\in \rr \text{ such that } f(n) \geq L + Kn^c \text{ for all } n \}			\\
		&= \inf \{ c \in [0,1] \mid \exists K>0,L\in \rr \text{ such that } f(n) \leq L + Kn^c \text{ for infinitely many } n \}		\\
		&= \displaystyle\vstr \liminf_{n \to \infty}  \frac{\ln f(n)}{\ln n}
\end{array}
\]
The constant $L$ is unnecessary. The exponents are obviously related by $\ssl{\phi} \leq \srl{\phi}$. Note that if $\ssl{\phi} <1$ then $f(n)$ is $\mathsf{o}(n)$ (since $f$ is sub-additive).
\begin{defi}
Let $B_n$ be the ball of radius $n$. Define
\[
\begin{array}{rllrll}
\gamma = \srl{\gamma} = \srl{\phi} & \text{for} & f(n) = -\ln P^{(n)}(e) &
\ssl{\gamma} = \ssl{\phi} & \text{for} & f(n) = -\ln P^{(n)}(e) \\
\nu = \srl{\nu} = \srl{\phi} & \text{for} & f(n) = \ln |B_n| &
\ssl{\nu} = \ssl{\phi} & \text{for} & f(n) = \ln |B_n| \\
\eta =\srl{\eta} = \srl{\phi} & \text{for} & f(n) = H(P^{(n)}) &
\ssl{\eta} = \ssl{\phi} & \text{for} & f(n) = H(P^{(n)}) \\
\srl{\beta} = \srl{\phi} & \text{for} & f(n) = \mathbb{E}|P^{(n)}| &
\beta = \ssl{\beta} = \ssl{\phi} & \text{for} & f(n) = \mathbb{E}|P^{(n)}| \\
\end{array}
\]
\end{defi}
Simple bounds between these quantities are explored in \S{}\ref{sint}.

\subsection{Off-diagonal decay}\label{ssoffdec}

An estimate which goes back to Carne \cite{Carne} and Varopoulos \cite{Var} on the ``off-diagonal'' behaviour of random walks is, for some $M,N >0$,
\stepcounter{teo}
\[ \tag{\text{\theteo}} \label{eqcarvar}
P^{(n)}(g) \leq Ne^{-M|g|^2/n}.
\]
Improvements of this theorem are known. For example, under a regularity hypothesis, there is a similar estimate due to Coulhon, Grigor'yan \& Zucca, see \cite[Theorem 5.2]{CGZ} but it concerns the ration $P^{(kn)}(g) / P^{(n)}(e)$ for some $k \geq 2$. When the group is of polynomial growth this actually implies \eqref{eqodd}.

It seems challenging to produce groups which violate the off-diagonal estimate from \eqref{eqodd}.
As pointed out in Dungey \cite[End of \S{}1]{Dung}, an interpolation argument shows this estimate is close to be true in all groups. More precisely: there are constant $M,N>0$ so that for any $\eps \in [0,1]$,
\stepcounter{teo}
\[ \tag{\text{\theteo}} \label{eqdung}
P^{(n)}(g) \leq  P^{(n)}(e)^{1-\eps} N^\eps e^{-\eps M |g|^2 /n }. 
\]
Note that the estimate \eqref{eqodd} is not the only estimate which would suffice for the proof of Theorem \ref{tconj-t}. The first obvious relaxation would be to have this estimate for $n < L |g|^{2-\eps}$ (for any $\eps>0$ with $L=L(\eps)$). 
The following condition would be also sufficient for the proof: for any $\eps >0$, there exists $n_0, K,L$ such that for any $n > L |g|^{2+\eps}$ and $|g| >n_0$, one has $\tfrac{P^{(2n)}(g)}{ P^{(2n)}(e)} \leq e^{-K |g|^2/n}$. 
Of course, any estimate with a fixed $\eps$ could also be of interest.

Recently, Brieussel \& Zheng \cite[Problem 9.3 and foregoing paragraphs]{BZ} have shown that there are groups for which the conclusion of Theorem \ref{tconj-t} cannot hold. They give a family of groups for which $\alpha < 1/2(1-\gamma)$. This implies that these groups violate \eqref{eqodd} (and its relaxations). B.~Vir\'ag pointed out to the author that the lamplighter on $\zz^3$ might also violate \eqref{eqodd} (by fine estimates on the return probability). Interestingly, all these groups do not have the Liouville property and their return exponent is $>1/2$.

\section{A lower bound using random walks}\label{slrw}

The idea will be to construct an equivariant uniform embedding of $G$ into $\jo{H} := \oplus_{n \in \nn} \ell^2 G$. The isometric action is simply the diagonal action of $G$ on each factor by the right-regular representation. The idea is to define a cocycle using a virtual coboundary of the form $w = \oplus a_n w_n$ where $w_n \in \ell^2G$ and $a_n \in \rr$. This yields a cocycle (in $\rho_{\ell^2G}^\nn$) if, for any $s \in S$,
\[
\|w - \rho_s^\nn v\|_2^2 
= \sum_n a_n^2 \|w_n - \rho_s w_n\|_2^2 < +\infty. 
\]
Simply put $a_n^2 = \max_{s \in S} \|w_n - \rho_s w_n\|_2^{-2} n^{-1-\eps}$, where $\eps>0$. 
The gradient of a function $f:G \to \rr$ is defined by $\nabla f(x,y) = f(y) - f(x)$ for two adjacent vertices $x,y$ in the Cayley graph. This operator is essentially build up by the various $f - \rho_s f$, and $\nabla f$ can be interpreted as a function $G \times S \to \rr$. The gradient is a bounded operator (since $S$ is finite) and its adjoint $\nabla^*$ can be used to form the Laplacian $\Delta$. These are related to $P$ by $\Delta = \nabla^*\nabla = |S|(I-P') = 2|S|(I-P)$.

Using that $\max_{s \in S} \|w_n - \rho_s w_n\|_2^2 \leq \|\nabla w_n\|_2^2$, one has 
\[
\|b(g)\|_2^2 
\geq \sum_{n \geq 1} n^{-1-\eps} \frac{\|w_n - \rho_g w_n\|_2^2}{\|\nabla w_n\|_2^2} 
\]
The idea will be to take $w_n = P^{(k_n)}$ for some $k_n \in [n,2n]$ (the  $k_n^\text{th}$-step distribution of a lazy random walk starting at $e \in G$). 
\begin{lem}\label{tnumer-l}
$\|P^{(n)} - \rho_g P^{(n)}\|_2^2 = P^{(2n)}(e) - P^{(2n)}(g)$.
\end{lem}
\begin{proof}
Indeed,
\[
\|P^{(n)} - \rho_g P^{(n)}\|_2^2 
= \pgen{ P^{(n)} - \rho_g P^{(n)} \mid P^{(n)} - \rho_g P^{(n)} } 
= 2 \|P^{(n)}\|^2 - 2 \pgen{ P^{(n)} \mid \rho_g P^{(n)} } 
\]
Since $S$ is symmetric, note that 
$\pgen{ f \mid P*g} = \pgen{P * f \mid g}$. 
Consequently,
\[
\pgen{ P^{(n)} \mid \rho_g P^{(n)} } = \pgen{ P^{(n)} \mid P^{(n)} * \delta_g } = \pgen{ P^{(2n)} \mid \delta_g} = P^{(2n)}(g).
\]
To get the claimed equality, use that, similarly, $P^{(2n)}(e) = \|P^{(n)}\|^2_2$.
\end{proof}
\begin{lem}\label{tdenom-l}
$\|\nabla P^{(n)}\|_2^2 = 2|S| \big( P^{(2n)}(e) - P^{(2n+1)}(e) \big)$
\end{lem}
\begin{proof} 
This is a simple calculation using the relation $\Delta = \nabla^* \nabla =  2|S|(I-P)$:
\[
\|\nabla P^{(n)}\|_2^2
= \langle \Delta P^{(n)} \mid P^{(n)} \rangle =2|S| \big(P^{(2n)}(e)- P^{(2n+1)}(e) \big). \qedhere
\]
\end{proof}
Putting Lemmas \ref{tnumer-l} and \ref{tdenom-l} together gives:
\[
2|S| \, \frac{\|P^{(n)} - \rho_g P^{(n)}\|^2}{\|\nabla P^{(n)} \|^2} = 
\frac{ P^{(2n)}(e) - P^{(2n)}(g)}{P^{(2n)}(e) - P^{(2n+1)}(e)} 
= \frac{ 1 - P^{(2n)}(g)/P^{(2n)}(e)}{1 - P^{(2n+1)}(e)/P^{(2n)}(e)}
\]
The next step is to find satisfying bounds for this quantity. There are reasonable estimates for the denominator, the following lemma is essentially from Tessera \cite[Proof of proposition 7.2]{Tessera}. For similar estimates on the entropy, see Erschler \& Karlsson \cite[Lemma 10]{EK} (see also Remark \ref{rmono} below). 
\begin{lem}\label{tdenomborn-l}
If $P^{(n)}(e) \geq e^{-f_P(n)}$ for a positive sub-additive increasing function $f_P$. Then, for any $n$ there is a $k \in [n,2n]$ 
\[
\bigg(1-  \frac{P^{(2k+1)}(e)}{P^{(2k)}(e)} \bigg) \leq 8 f_P(n)/n.
\]
\end{lem}
\begin{proof}
Let $F(n) = -\ln P^{(n)}(e)$. Let $C_n$ be the largest real number such that, for any $q \in [n, 2n]$.
\[
F(q+1) - F(q) \geq C_n f_P(n)/n.
\]
This implies $F(2n)-F(n) \geq C_n f_P(n)$, and in particular $F(2n) \geq C_n f_P(n)$ (since $F(n) \geq 0$). By hypothesis, $F(2n) \leq f_P(2n) \leq 2 f_P(n)$ so that $C_n \leq 2$. Thus, for any $n$, there exists a $k \in [n,2n]$ such that $F(k+1) - F(k) \leq 2 f_P(n)/n$. This implies
\[
1- \frac{P^{(k+1)}(e)}{P^{(k)}(e)} \leq 1- e^{-2f_P(n)/n} \leq \frac{2f_P(n)}{n},
\]
where the last inequality comes from $1-e^{-x} \leq x$ for $x \geq 0$.

The actual statement is obtained by doing the same argument with $G(n) = F(2n)$ and noticing that an additional constant comes in since one then looks at the gradient defined for the generating set $S' = S^2$.
\end{proof}

\begin{proof}[Proof of Theorems \ref{tlrwcomp-t} and \ref{tconj-t}]
Using $w_n = P^{(k_n)}$ where $k_n \in [n,2n]$ is given by Lemma \ref{tdenomborn-l} and the bound mentioned above for the numerator, one finds (using $1 \leq \tfrac{k_n}{n} \leq 2$)
\[
\|b(g)\|_2^2 \geq \sum_{n \geq 1} K'' n^{-\gamma-\eps} (1 - P^{(2n)}(g)/P^{(2n)}(e) ) 
\]
So the question boils down to showing for which $n$ one has, $\frac{P^{(2n)}(g)}{P^{(2n)}(e)} \leq 1/2$.

For example, assuming \eqref{eqodd} holds, one 
sees
this is true for $n \leq M' |g|^2/ \ln (2N)$ (since, necessarily $N \geq 1$). Hence, restricting the sum to those values of $n$: 
\[
\|b(g)\|_2^2 
\geq \sum_{n \leq M'|g|^2 / \ln(2N)} \tfrac{K''}{2} n^{-\gamma-\eps}
\geq \widetilde{K} |g|^{2(1-\gamma-\eps)}.
\]
Letting $\eps \to 0$ proves Theorem \ref{tconj-t} (even though the constant gets worse as $\eps \to 0$).

Using \eqref{eqcarvar} instead of 
\eqref{eqodd}, 
one must restrict the sum to $n < K' |g|^{2/(1+\gamma)}$. This yields a weaker lower bound of $\alpha \geq \frac{1-\gamma}{1+\gamma}$ (but is true in any group) and proves Theorem \ref{tlrwcomp-t}.
\end{proof}
If the reader is interested in compression functions (rather than exponents), then it is fairly easy to check that, given $f_P$ as in Lemma \ref{tdenomborn-l}, $\rho_-(k) \geq k^{1/(1+\gamma)} / f_P( k^{2/(1+\gamma)} )^{1/2}$ and, if \eqref{eqodd} holds, $\geq k/ f_P(k^2)^{1/2}$.

\begin{rem}
It would be interesting to generalise this proof by picking $v_n$ elements which are in $V_{\lambda_n}$ with $\lambda_n \to 0$, where $V_\lambda$ is the image of the spectral projection (of the Laplacian) to eigenvalues $\leq \lambda$. This would ensure a good bound for the denominator. For the numerator, one needs to elucidate how to relate bound on the von Neumann dimension of $V_\lambda$ to upper estimates on $\langle v \mid \rho_g v \rangle$ for $v \in V_\lambda$.

More precisely, if $\lambda_n = 1/n$ and $v_n \in V_{1/n}$ then one would require\\
$\cdot$ either, for some $K>0$, $\langle v_n \mid \rho_g v_n \rangle \leq 1/2$ when $|g|^{2-2\gamma} > K n$; \\ 
$\cdot$ or, for some $K,K'>0$ and $\eps >0$, $\langle v_n \mid \rho_g v_n \rangle \leq \mathrm{exp}(-K |g|^{2-2\gamma}/n)$ when $n > K'|g|^{2-2\gamma +\eps}$. 

Using the results of Bendikov,Pittet \& Sauer \cite{BPS}, note that $P^{(n)}(e) \succcurlyeq \mathrm{exp}(-n^\gamma)$ (near infinity) corresponds to the fact that the von Neumann dimension of $V_\lambda \succcurlyeq \mathrm{exp}(-\lambda^{\gamma/(1-\gamma)}$ (near zero).
\hfill $\Diamond$
\end{rem}

\begin{rem} \label{rmono}
There is an alternative proof of Lemma \ref{tdenomborn-l} along the lines of Erschler \& Karlsson \cite[Lemma 10]{EK}. Let $F(n) = -\ln P^{(2n)}(e)$. Then it is well-known that $F(n+1)-F(n)$ is decreasing, see Woess' book \cite[(10.1) Lemma]{Woe}. 
\hfill $\Diamond$
\end{rem}

\section{Some relations between the exponents}\label{sint}

The aim of this section is to relate the return, speed, entropy and growth exponents. An elementary computation (see Avez \cite[Theorem 3]{Avez74}) shows, using concavity of $\ln$, that
\stepcounter{teo}
\[\tag{\text{\theteo}} \label{etagamma}
H(P^{(n)}) 
\geq -\ln \Big(\sum_{g \in G} P^{(n)}(g)^2 \Big)  
= -\ln \|P^{(n)}\|_2^2 
= -\ln P^{(2n)}(e).
\]
Hence, $\ssl{\gamma} \leq \ssl{\eta}$ and $\srl{\gamma} \leq \srl{\eta}$. (With Kesten's criterion \cite{Kes2}, this shows Liouville$\implies$amenable.)

\eqref{eqcarvar}, gives $P^{(n)}(g) \leq Ne^{-M|g|^2/n}$. 
This, together with convexity of $x \mapsto x^2$, gives another useful bound, found in either Amir \& Vir\'ag \cite[Proposition 8]{AV} or Erschler \cite[Lemma 7.(i)]{Ers}:\stepcounter{teo}
\[ \tag{\text{\theteo}} \label{AVE}
H(P^{(n)}) 
\geq \ln N + M\sum_{g \in G} P^{(n)}(g) \tfrac{|g|^2}{n} 
\geq \ln N + \tfrac{M}{n} (\mathbb{E}|P^{(n)}|)^2. 
\]
Thanks to Erschler \& Karlsson \cite[Corollary 9.ii]{EK}, this inequality is also true for measures with finite second moment. This implies that $\ssl{\beta} \leq \frac{1+\ssl{\eta}}{2}$ and $\srl{\beta} \leq \frac{1+\srl{\eta}}{2}$ and constitutes a proof of (ii)$\implies$(iii) in the equivalences of the Liouville property described in \S{}\ref{ssexp}. 

There is also ``classical'' bound obtained by Varopoulos' method (see \eg Woess' book  \cite[(14.5) Corollary]{Woess}) relating growth and return exponent: $\ssl{\gamma} \geq \frac{\ssl{\nu}}{2+\ssl{\nu}}$. 

The following lemma (see \eg \cite[\S{}1.2]{BL}) will be useful.
\begin{lem}\label{tcohu-l}
Let $f:\nn \to \rr_{\geq 0}$ be a sub-additive, non-decreasing function with $f(0)=0$. If $g$ is the concave hull of $f$ then $f(x) \leq g(x) \leq 2 f(x)$.
\end{lem}

The upcoming lemma is an improvement of a standard inequality (see \eg Erschler \cite[Lemma 6]{Ers}) and of the simple inequality $H(P^{(n)}) \leq \ln |B_n|$ (see Erschler \& Karlsson \cite[Lemma 1]{EK}). Since it might be of larger use, it will be stated in full generality, namely $P$ will be some measure and $S^*$ some finite (symmetric) generating set.
\begin{lem}\label{tvitinf-l}
Let $|g|_*$ be the word length for $S^*$. Assume $P$ has finite first moment (\ie $\sum_{g \in G} P(g) |g|_* < +\infty$ ), and $B_n = \{ g \in G \mid |g|_* \leq n\}$. Let $|B_n| = e^{f_V(n)}$ and assume $|B_n|$ is at least quadratic in $n$. Then
\[
H(P^{(n)}) \leq L + 4 f_V(\mathbb{E}|P^{(n)}|_*). 
\]
In particular, $\ssl{\beta} \srl{\nu} \geq \ssl{\eta}$, $\srl{\beta} \ssl{\nu} \geq \ssl{\eta}$  and $\srl{\beta} \srl{\nu} \geq \srl{\eta}$.
\end{lem}
\begin{proof}
The idea is to compare a measure $m$ to a measure $m'$ which is uniform on spheres. First,
\[
H(m) - \sum_{g \in G} m(g) \ln(\frac{1}{m'(g)}) = \sum_{g \in G} m(g) \Big( -\ln \frac{m(g)}{m'(g)} \Big) \leq 0
\]
using $-\ln t \leq \tfrac{1}{t} -1$. Now let $a_i = |\delta B_i|$ where $\delta B_i = B_i \setminus B_{i-1}$ and $B_{-1} = \emptyset$ and $m'(g) = \phi(|g|_*)/ a_{|g|_*}$ where $\phi(k) = L_1 |B_k|^{-1}$ and $L_1$ chosen so that $\sum_{k \geq 0} \phi(k) = 1$. Then,
\[
H(m) \leq \sum_{g \in G} m(g) \big( \ln a_{|g|_*} - \ln \phi(|g|_*) \big).
\]
Then, one has (with $L'= \ln(L_1)$)
\[
H(m) \leq L' + 2 \sum_{g \in G} m(g) f_V(|g|_*) \leq L' + 4 f_V\Big( \sum_{g \in G} m(g) |g|_* \Big)
\]
by passing to the concave hull of $f_V$ and using Lemma \ref{tcohu-l} to bound this by $2f_V$. 
This shows $H(P^{(n)}) \leq L' + 4 f_V(\mathbb{E}|P^{(n)}|_*)$, as desired.

The bound $\srl{\eta} \leq \srl{\beta} \srl{\nu}$ follows directly while the others follow by applying the inequality for infinitely many $n$.
\end{proof}
If one assumes $|B_n| \leq Le^{Kn^\nu}$, one can also obtain the statement $H(P^{(n)}) \leq L' + K' (\mathbb{E}|P^{(n)}|_*)^\nu$ with $K'$ as close as desired to $K$, as in Erschler \& Karlsson \cite[Lemma 1]{EK}. 
\begin{cor}\label{tresum-c}
Assume $|B_n| \leq Le^{Kn^\nu}$ and $|B_n|$ is more than quadratic. For any measure of finite second moment (\ie $\sum_{g \in G} P(g) |g|^2 < +\infty$), one has
\begin{itemize}\setlength{\itemsep}{0ex} \renewcommand{\labelitemi}{$\cdot$}
 \item $\mathbb{E}|P^{(n)}| \leq K'n^{1/(2-\nu)}$, 
\item $H(P^{(n)}) \leq L'' + K'' n^{\nu /(2-\nu)}$, 
\item and $P^{(2n)}(e) \geq \mathrm{exp}\big(-H(P^{(n)})\big) \geq L'' \mathrm{exp}(- K'' n^{\nu /2-\nu})$. 
\end{itemize}
In particular,
\[
 \ssl{\beta} \leq \srl{\beta} \leq \tfrac{1}{2-\srl{\nu}}
\quad \text{ and } \quad 
 \srl{\gamma} \leq \srl{\eta} \leq \srl{\beta} \srl{\nu} \leq \frac{\srl{\nu}}{2-\srl{\nu}}.
\]
\end{cor}
\begin{proof}
Using first \eqref{AVE} (which extends to measures of finite second moment by Erschler \& Karlsson \cite[Corollary 9.ii]{EK}) then Lemma \ref{tvitinf-l}, one has 
$( \mathbb{E}|P^{(n)}| )^2 \leq n \big( \tilde{L} + 4(\ln K) (\mathbb{E}|P^{(n)}|)^\nu \big)$. Putting $K' = \big( 4 \ln K+ \tilde{L}/\mathbb{E}|P^{(1)}|^\nu \big)^{1/(2-\nu)}$, this implies the first claim. The second claim is obtained by concatenating 
Lemma \ref{tvitinf-l} and the bound on speed just obtained. The relation \eqref{etagamma} is also used in the sequence of inequalities in term of exponents.
\end{proof}
Lemma \ref{tvitinf-l} and Corollary \ref{tresum-c} 
finish the proof of Theorem \ref{tint-t}. 

Let us mention an additional inequality. This inequality is already present in Coulhon \& Grigoryan \cite[\S{}6]{CG} in a sharper form but with extra hypothesis. The proof presented here is elementary if one knows \eqref{eqdung} and could be improved in the case of polynomial growth (though it does not meet \cite{CG}).
\begin{lem}\label{tgamnupasu-l}
Assume $|B_n|=e^{f_V(n)}$ is at least cubic. Let $f$ be the concave hull of $f_V$, and $F$ the inverse function of [the strictly increasing function] $k \mapsto k^2/f(k)$. Then $P^{(n)}(e) \geq K'' |B_{F(L'' n)}|^{-2}  F(L'' n)^{-1}$ for some $K'',L''>0$.
\end{lem}
\begin{proof}
Write $|B_n| = e^{f_V(n)}$ as before.Then, using the bound \eqref{eqdung} one has, for any $\eps \in ]0,1[$, 
\[
1 = \sum_{g \in G} P^{(n)}(g) \leq \sum_{k=0}^n |B_k| P^{(n)}(e)^{1-\eps} N^\eps e^{-M \eps k^2/n} \leq P^{(n)}(e)^{1-\eps} \sum_{k=0}^n N^\eps e^{f(k) - M \eps k^2/n},
\]
where $f$ is the concave hull of $f_V$.
Let $n_0 = \inf\{ k \mid k^2 /f(k) \geq n /M\eps \}$. Note that $k \mapsto k^2/f(k)$ is strictly increasing. Indeed, since $f$ is concave and $f(0)=0$ one has $f(n) = \sum_{i=1}^n f(i)-f(i-1) \geq n \big( f(n) -f(n-1) \big)$. That $(k+1)^2/f(k+1) > k^2/f(k)$ then follows from:
\[
k^2 \big( f(k+1) - f(k) \big) \leq \tfrac{k^2  f(k+1) }{k+1} < k \big( f(k) + f(1) \big)  \leq 2 k f(k) < (2k+1) f(k).
\]
Hence, the exponent of the exponential is negative if $k \geq n_0$.
Since $P^{(n)}(e)^{1-\eps} n \to 0$ for some $\eps \in ]0,1[$ (because $|B_n| \geq Kn^3$ implies $P^{(n)}(e) \leq K'n^{3/2}$), one may write (with $\delta_n \to 0$ as $n \to \infty$)
\[
1-\delta_n \leq P^{(n)}(e)^{1-\eps} \sum_{k=0}^{n_0} N^\eps e^{f(k) - M \eps k^2/n} \leq P^{(n)}(e)^{1-\eps} \sum_{k=0}^{n_0} N^\eps e^{f(k)} \leq n_0 P^{(n)}(e)^{1-\eps} e^{f(n_0)}.
\]
This implies that $P^{(n)}(e) \geq K'' e^{- f(n_0)} n_0^{-1}$. To conclude apply Lemma \ref{tcohu-l}: $f(x) \leq 2f_V(x)$.
\end{proof}
The preceding lemma implies $\srl{\gamma} \leq \srl{\nu} / (2-\srl{\nu})$ and $\ssl{\gamma} \leq \ssl{\nu} / (2-\srl{\nu})$, but these inequalities already follows for a larger class of measures from \eqref{etagamma} and Corollary \ref{tresum-c}. One cannot deduce $\ssl{\gamma} \leq \ssl{\nu} / (2-\ssl{\nu})$ from Lemma \ref{tgamnupasu-l}.

Lastly, the estimate $\srl{\gamma} \geq \frac{\srl{\nu}}{2+\srl{\nu}}$ can be deduced from Coulhon, Grigoryan \& Pittet \cite[Corollary 7.2]{CGP}.
The estimates cited or proved in this paper can also be summed up by:
\[
\beta \overset{ii}{\leq} \frac{1+\eta}{2}
\hspace{1ex} , \hspace{1ex}
\frac{\ssl{\nu}}{2+\ssl{\nu}} \leq \ssl{\gamma} 
\overset{i}{\leq} \ssl{\eta} 
\overset{i}{\leq} \min(\ssl{\beta} \srl{\nu} , \srl{\beta} \ssl{\nu} )
\overset{ii}{\leq} \frac{\ssl{\nu}}{2-\srl{\nu}}
\hspace{1ex} \text{and} \hspace{1ex} 
\frac{\srl{\nu}}{2+\srl{\nu}}
\leq \srl{\gamma} \overset{i}{\leq} \srl{\eta} \overset{i}{\leq} \srl{\beta} \srl{\nu} \overset{ii}{\leq} \frac{\srl{\nu}}{2-\srl{\nu}}
\]
where $i$ (resp. $ii$) denotes inequality which hold for measures with finite first (resp. second) moment, the remaining inequalities hold only for finitely supported measures and the absence of bars [above or below] the exponent mean it holds if bars are put on both sides at the same place.

The lower bound $\ssl{\beta} \geq \ssl{\nu}/ \srl{\nu}(2+\ssl{\nu})$ is not optimal (B.~Vir\'ag gave a [sharp] lower bound of $\tfrac{1}{2}$; see Lee \& Peres \cite{LP}).

Other inequalities which could be interesting to explore are: 
$\ssl{\eta} \leq \frac{\ssl{\nu}}{2-\ssl{\nu}}$? 
$\ssl{\gamma} \leq \frac{\ssl{\nu}}{2-\ssl{\nu}}$? 
A more interesting one (since a positive answer combined with \eqref{AVE} would give a proof of Conjecture \ref{laconj}) is
\begin{ques}\label{laques}
Does the inequality 
$\ssl{\eta} \leq \frac{\srl{\gamma}}{1-\srl{\gamma}}$ 
hold? Could it even hold for all measures with finite second moment?
\end{ques}
This has been answered in the positive by Saloff-Coste \& Zheng \cite[Theorem 1.8]{SCZ}. 

Let us conclude with this possibly well-known lemma.
\begin{lem}\label{tvitsurj-l}
Assume $\psi:G \twoheadrightarrow H$ is a surjective homomorphism. Let $S = \mathrm{Supp} P$ be generating for $G$ (hence $\psi(S)$ generates $H$). Let $P' = \psi^*P$, \ie $P'(A) =  P\big( \psi^{-1}A \big)$.
Then 
$\mathbb{E}|P^{(n)}_{e_G}| \geq \mathbb{E}|P'^n_{e_H}|$ (where the word lengths $|\cdot|$ are for $S$ and $\psi(S)$ respectively).
\end{lem}
\begin{proof}
Let $d_H$ be the distance of the Cayley graph with respect to $S_H =$ support of $P'$. Define the function $d':G \to \nn$ by $d'(\gamma) = d_H \big( \psi(\gamma),e_H \big)$. Note that $d'(\gamma) \leq d_G(\gamma,e)$: indeed $d_H(h_1,h_2) = d_G\big( \psi^{-1}(h_1), \psi^{-1}(h_2) \big)$, so that $d'(\gamma) = d_G(\gamma N, N)$ where $N = \ker \psi$. Let $W^G_n$ be the random walker on $G$ and $W^H_n$ be the random walker on $H$ (which moves according to $P'$ as in the statement). Note that $\mathbb{P} \big( d_H(W^H_n,e_H) = i \big) = \mathbb{P} \big( d'(W^G_n) = i\big)$. This implies
\[
\mathbb{E}|P'^n| = \mathbb{E} \big( d_H(W_n^H,e_H) \big) = \mathbb{E} \big( d'(W_n^G) ) \leq \mathbb{E} \big( d_G(W_n^G, e_G) \big) = \mathbb{E}|P^{(n)}| \qedhere
\]
\end{proof}
In particular, this proves that $\mathbb{E}|P^{(n)}| \geq K_P n^{1/2}$ for any $G$ with a non-trivial homomorphism to $\zz$ (this is true for any group, due to Vir\'ag, see \cite{LP}).

The statement of Lemma \ref{tvitsurj-l} may be generalised to coverings of graphs and more general maps. 
Here is a classical example. Define ``levels'' in the $k$-regular tree by looking at points which are at the same distance to some [fixed] point at infinity. The ``level maps'' gives a morphism from the tree to the line $\zz$. The arguments of the above Lemma apply to this map, but with a biased random walk on $\zz$. This gives a rather precise estimate of the speed.  


\section{Some known values}\label{stab}

Below is a table containing cases where $\alpha, \ssl{\beta}$ and $\srl{\gamma}$ are known. The convention for wreath products $L \wr H$ is that $L$ is the ``lamp state'' group, \eg $\zz_2 \wr \zz$ is the usual lamplighter on the line.
One could complete the table for many other wreath products using Naor \& Peres \cite[Theorem 6.1]{NPspeed}, Naor \& Peres \cite[Theorem 3.1]{NPlamp}, Pittet \& Saloff-Coste \cite[Theorem 3.11 and Remark (ii) after Theorem 3.15]{PSCexp} and Revelle \cite[Theorem 1]{Revelle}.

The lower bound of Theorem \ref{tconj-t} [assuming \eqref{eqodd} holds] meets compression in (A), (C), (D) if $d=2$, (E) if $H$ has polynomial growth, (H) and (I). It also meets speed, except in the last two cases. The lower bound meets neither speed nor compression in (B), (E) if $H$ is polycyclic [since $\srl{\gamma} = \tfrac{3}{5}$] and (F) if $k \geq 3$. All the groups mentioned that have $\srl{\gamma} = \tfrac{1}{2}$ are Liouville.

%

\newcommand{\fn}[1]{\footnotemark[#1]}
\renewcommand{\thefootnote}{(\alph{footnote})}
\newcounter{toto}
\newcommand{\ft}[1]{\setcounter{toto}{#1}(\alph{toto})}
\newcounter{trala}
\setcounter{trala}{0}
\newcommand{\tra}[0]{\stepcounter{trala}\Alph{trala}:~ }
\begin{table}[H]
\begin{tabular}{@{}r@{}l@{}|cc@{}|@{}ccc@{}}
& Group 		& $\ssl{\beta}$ 	& $\srl{\gamma}$	& $1-\srl{\gamma}$ 	& $\alpha$		& $1/ 2 \ssl{\beta}$ \\[.3ex]
\hhline{==|==|===}
&			&			&			&			&			&	\\[-2.3ex]
\tra&Polynomial growth	& $\tfrac{1}{2}$\fn{1}	& $0$\fn{7}		& $1$			& $1$\fn{2}		&$1$ \\[.2ex]
\hline
\tra&Polycyclic of	&\multirow{3}{*}{$\tfrac{1}{2}$\fn{2}}
						&\multirow{3}{*}{$\tfrac{1}{3}$\fn{8}}
									&\multirow{3}{*}{$\tfrac{2}{3}$}
												&\multirow{3}{*}{$1$\fn{2}}&\multirow{3}{*}{$1$} \\
&exponential growth	&			&			&			&			&	\\
&or $F \wr \zz$ with $F$ finite&		&			&			&			&	\\
\hline
&			&			&			&			&			&	\\[-2.3ex]
\tra&$\zz \wr \zz$	&$\tfrac{3}{4}$\fn{3} 	& $\tfrac{1}{3}$\fn{9}	& $\tfrac{2}{3}$	& $\tfrac{2}{3}$\fn{11}	& $\tfrac{2}{3}$ \\[.5ex]
\hline
\tra&$F \wr H$ with $F$ finite or 
			&\multirow{3}{*}{$1$\fn{4}}& \multirow{3}{*}{$\tfrac{d}{d+2}$\fn{9}}
									&\multirow{3}{*}{$\tfrac{2}{d+2}$}
												& \multirow{3}{*}{$\tfrac{1}{2}$\fn{12}}
															& \multirow{3}{*}{$\tfrac{1}{2}$} \\
&$\zz$ and $H$ polynomial &			&			&			&			&	\\
&growth of degree $d \geq 2$&			&			&			&			&	\\
\hline
&			&			&			&			&			&	\\[-2.3ex]
\tra&$H \wr \zz^2$ with $H$ amenable&\multirow{2}{*}{$1$\fn{4}}
						&\multirow{2}{*}{$\geq \tfrac{1}{2}$\fn{9}}
									&\multirow{2}{*}{$\leq\tfrac{1}{2}$}
												&\multirow{2}{*}{$\tfrac{1}{2}$\fn{13}}
															&\multirow{2}{*}{$\tfrac{1}{2}$}\\
&and $\alpha(H) \geq \tfrac{1}{2}$&		&			&			&			&\\[.5ex]
\hline
\tra&$( \ldots ((\zz \wr \zz) \wr \zz)\ldots) \wr \zz$
			&\multirow{3}{*}{$1-\tfrac{1}{2^{k}}$\fn{3}}
						& \multirow{3}{*}{$\tfrac{k-1}{k+1}$\fn{9}}
									&\multirow{3}{*}{$\tfrac{2}{k+1}$}
												& \multirow{3}{*}{$\tfrac{1}{2-2^{1-k}}$\fn{11}}
															& \multirow{3}{*}{$\tfrac{1}{2-2^{1-k}}$} \\ 
&iterated wreath product&			&			&			&			&\\
&with $k$ ``$\zz$'', $k \geq 1$&		&			&			&			&\\
\hline
\tra&Intermediate growth&\multirow{2}{*}{$[\tfrac{1}{2}, \tfrac{1}{2-\nu}]$\fn{5}}
						&\multirow{2}{*}{$[\tfrac{v}{2+v}, \tfrac{\nu}{2-\nu}]$\fn{10}}
									&\multirow{2}{*}{$[\tfrac{1-\nu}{1-\nu/2}, \tfrac{1}{1+v/2}]$}
												&\multirow{2}{*}{?}	&\multirow{2}{*}{$[1-\tfrac{\nu}{2}, 1]$} \\
&$e^{n^v} \preceq |S^n| \preceq e^{n^\nu}$  &	&			&			&			&\\
\hline
\tra&``Incompressible''	&\multirow{2}{*}{?}&\multirow{2}{*}{$1$}&\multirow{2}{*}{$0$}&\multirow{2}{*}{$0$\fn{14}}&\multirow{2}{*}{$\geq \tfrac{1}{2}$}\\
&amenable groups	&			&			&			&			&\\
\hline
			&			&			&			&			&\\[-2.3ex]
\tra&Property $(T)$ groups& $1$\fn{6}		& $1$\fn{6}		& $0$			& $0$\fn{6}		& $\tfrac{1}{2}$ \\[.2ex]
\end{tabular}
\end{table}


\begin{table}[H]
{\footnotesize
\begin{tabular}{r@{$\,$}l}
      & {\bf Table's references}\\
\ft{1}& \multirow{2}{*}{\parbox{.85\textwidth}{The upper bound is classical; see \S{}\ref{sint}. The (general) matching lower bound is due to 
	Vir\'ag (see Lee \& Peres \cite{LP}); this particular instance could be obtained by arguments of \S{}4.} }\\
      & \\
\ft{2}& \multirow{2}{*}{\parbox{.85\textwidth}{The value of compression (from Tessera \cite[Theorems 9 and 10]{Tessera}) imply the value of speed. For finer estimates on 
	speed see Thompson \cite[Theorem 1]{Thomp}.}} \\
      & \\
\ft{3}& This may be found either in Erschler \cite[Theorem 1]{Ers} or Revelle \cite[Theorem 1]{Revelle}.\\
\ft{4}& See Erschler \cite[Theorem 1]{Ers} or Naor \& Peres \cite[Theorem 6.1]{NPspeed}.\\
\ft{5}& \multirow{2}{*}{\parbox{.85\textwidth}{The upper bound is easy; see \S\ref{sint}. The lower bound is the general one due to Vir\'ag, see 
	the introduction of Lee \& Peres \cite{LP}.}}\\
      & \\
\ft{6}& \multirow{3}{*}{\parbox{.85\textwidth}{Kesten's criterion for amenability \cite{Kes2} shows $\gamma=1$, use Kesten \cite[Theorem 5]{Kes1} or Lemma \ref{tvitinf-l} 
	to get $\beta =1$. Property (T) groups do not have the Haagerup property. In particular, they have no proper affine action on a Hilbert space; hence $\alpha =0$.}}\\
      & \\
      & \\
\ft{7}& \multirow{2}{*}{\parbox{.85\textwidth}{$0$ should be interpreted as arbitrarily small. This is the classical estimate of Varopoulos, see Woess' book 
	\cite[(14.5) Corollary]{Woess}.}} \\
      & \\
\ft{8}& Due to Varopoulos; see \cite[\S{}1.1]{PSCexp} for a list of possible references.\\
\ft{9}& See Pittet \& Saloff-Coste \cite[Theorems 3.11 and 3.15]{PSCexp}\\
\ft{10}& \multirow{2}{*}{\parbox{.85\textwidth}{For the lower bound see Woess' book \cite[(14.5) Corollary]{Woess}. The upper bound is 
	Coulhon, Grigor'yan \& Pittet \cite[Corollary 7.4]{CGP}; see also \S{}\ref{sint} of the present text.}}\\
      & \\
\ft{11}& See Naor \& Peres \cite[Corollary 1.3]{NPspeed}.\\
\ft{12}& See Naor \& Peres \cite[Theorem 3.1]{NPlamp}.\\
\ft{13}& See Naor \& Peres \cite[Remark 3.4]{NPspeed}.\\
\ft{14}& See Austin \cite{Austin} or Bartholdi \& Erschler \cite[\S{}1.2 and \S{}7]{BE}. $\alpha=0$ implies $\gamma=1$.
\end{tabular}
}
\end{table}

\vfill

Except in (H) and (I), the upper bound $\alpha \leq 1/2\ssl{\beta}$ of Naor \& Peres \cite{NPspeed} meets compression. ``Incompressible'' (\ie of compression exponent $0$) amenable groups were first constructed by Austin (a solvable group, see \cite{Austin}) and, more recently, Bartholdi \& Erschler \cite[\S{}1.2 and \S{}7]{BE}. It seems reasonable to believe there is an amenable group where the compression meets neither the upper bound of \cite{NPspeed} nor the lower bound of Theorem \ref{tconj-t} [assuming \eqref{eqodd} holds].


\begin{thebibliography}{10}

\bibitem{Amir}
G.~Amir,
\newblock On the joint behaviour of speed and entropy of random walks on groups,
\newblock arXiv:1509.00256

\bibitem{AV}
G.~Amir and B.~Vir{\'a}g,
\newblock Speed exponents for random walks on groups,
\newblock arXiv:1203.6226 


\bibitem{AGS}
G.~N.~Arzhantseva, V.~S.~Guba and M.~V.~Sapir,
\newblock Metrics on diagram groups and uniform embeddings in a Hilbert space,
\newblock \emph{Comment. Math. Helv.} \textbf{81}(4):911--929, 2006. 

\bibitem{Austin}
T.~Austin,
\newblock Amenable groups with very poor compression into Lebesgue spaces,
\newblock \emph{Duke Math. J.} \textbf{159}(2):187--222, 2011. 

\bibitem{ANP}
T.~Austin, A.~Naor and Y.~Peres,
\newblock The wreath product of $\zz$ with $\zz$ has Hilbert compression exponent 2/3, 
\newblock \emph{Proc. Amer. Math. Soc.} \textbf{137}(1):85--90, 2009.

\bibitem{Avez72}
A.~Avez,
\newblock Entropie des groupes de type fini,
\newblock \emph{C. R. Acad. Sci. Paris S\'er. A-B} \textbf{275}:A1363--A1366, 1972.

\bibitem{Avez74}
A.~Avez,
\newblock Th\'eor\`eme de Choquet-Deny pour les groupes \`a croissance non exponentielle,
\newblock \emph{C. R. Acad. Sci. Paris S\'er. A} \textbf{279}:25--28, 1974. 

\bibitem{BE}
L.~Bartholdi and A.~Erschler,
\newblock Imbeddings into groups of intermediate growth,
\newblock arXiv:1403.5584

\bibitem{BCV}
B.~Bekka, A.~Ch\'erix and A.~Valette,
\newblock Proper affine isometric actions of amenable groups,
\newblock in \emph{Novikov conjectures, index theorems and rigidity, Vol.
2} (Oberwolfach, 1993), volume 227 of London Math. Soc. Lecture Note Ser.,
pages 1--4. Cambridge Univ. Press, Cambridge, 1995.

\bibitem{BPS}
A.~Bendikov, C.~Pittet and R.~Sauer,
\newblock Spectral distribution and $L^2$-isoperimetric profile of Laplace operators on groups,
\newblock \emph{Math. Ann.} \textbf{354}:43--72, 2012.

\bibitem{BL}
Y.~Benyamini and J.~Lindenstrauss,
\newblock {\em Geometric nonlinear functional analysis. {V}ol. 1}, volume~48 of
  {\em American Mathematical Society Colloquium Publications},
\newblock American Mathematical Society, Providence, RI, 2000.

\bibitem{BZ}
J.~Brieussel and T.~Zheng,
\newblock Speed of random walks, isoperimetry and compression of finitely generated groups,
\newblock arXiv:1510.08040.

\bibitem{Car}
M.~Carette (appendix by S.~Arnt, T.~Pillon and A.~Valette),
\newblock The Haagerup property is not invariant under quasi-isometry,
\newblock arXiv:1403.5446 

\bibitem{Carne}
T.~K.~Carne, 
\newblock A transmutation formula for Markov chains, 
\newblock \emph{Bull. Sci. Math. (2)} \textbf{109}(4):399--405, 1985.

\bibitem{CG}
T.~Coulhon and A.~Grigor'yan,
\newblock On-diagonal lower bounds for heat kernels and Markov chains,
\newblock \emph{Duke Math. J.} \textbf{89}(1):133--199, 1997. 

\bibitem{CGP}
T.~Coulhon, A.~Grigor'yan and C.~Pittet,
\newblock A geometric approach to on-diagonal heat kernels lower bounds on groups,
\newblock \emph{Ann. Inst. Fourier} \textbf{51}:1763--1827, 2001.

\bibitem{CGZ}
T.~Coulhon, A.~Grigor'yan and F.~Zucca ,
\newblock The discrete integral maximum principle and its applications,
\newblock \emph{Tohoku Math. J.} \textbf{57}(4):447--621, 2005.

\bibitem{Dung}
N.~Dungey,
\newblock Properties of random walks on discrete groups: Time regularity and off-diagonal estimates,
\newblock \emph{Bull. Sci. math.} \textbf{132}:359--381, 2008.
 
\bibitem{Ers}
A.~Erschler, 
\newblock On drift and entropy growth for random walks on groups,
\newblock \emph{Ann. Probab.} \textbf{31}(3):1193--1204, 2003.


\bibitem{EK}
A.~Erschler and A.~Karlsson,
\newblock Homomorphisms to $\rr$ constructed from random walks,
\newblock \emph{Ann. Inst. Fourier (Grenoble)} \textbf{60}(6):2095--2113, 2010.

\bibitem{GK}
E.~Guentner and J.~Kaminker,
\newblock Exactness and uniform embeddability of discrete groups,
\newblock \emph{J. London Math. Soc. (2)} \textbf{70}(3):703--718, 2004. 

\bibitem{Kai} 
V.~Kaimanovich,
\newblock Boundary behaviour of Thompson's group,
\newblock \emph{In preparation}.


\bibitem{Kes1}
H.~Kesten,
\newblock Symmetric random walks on groups,
\newblock \emph{Trans. Amer. Math. Soc.} \textbf{92}:336--354, 1959.

\bibitem{Kes2}
H.~Kesten,
\newblock Full Banach mean values on countable groups, 
\newblock \emph{Math. Scand.}, \textbf{7}:146--156, 1959.

\bibitem{KV}
M~Kotowski and B.~Vir\'ag,
\newblock Non-Liouville groups with return probability exponent at most 1/2,
\newblock arXiv:1408.6895.

\bibitem{LP}
J.~Lee and Y.~Peres,
\newblock Harmonic maps on amenable groups and a diffusive lower bound for random walks,
\newblock \emph{Ann. Probab.} \textbf{41}(5):3392--3419, 2013.

\bibitem{NPspeed}
A.~Naor and Y.~Peres,
\newblock Embeddings of discrete groups and the speed of random walks,
\newblock \emph{Int. Math. Res. Not.} IMRN 2008, Art. ID rnn 076, 34 pp. 

\bibitem{NPlamp}
A.~Naor and Y.~Peres,
\newblock $L_p$-compression, traveling salesmen, and stable walks,
\newblock \emph{Duke Math. J.} \textbf{157}(1):53--108, 2011.

\bibitem{PSCstab}
C.~Pittet and L.~Saloff-Coste, 
\newblock On the stability of the behavior of random walks on groups,
\newblock \emph{J. Geom. Anal.} \textbf{10}(4):713--737, 2000.

\bibitem{PSCexp}
C.~Pittet and L.~Saloff-Coste,
\newblock On random walks in wreath products,
\newblock \emph{Ann. Probab.} \textbf{30}(2):948--977, 2002.


\bibitem{Revelle}
D.~Revelle,
\newblock Rate of escape of random walks on wreath products and related groups,
\newblock \emph{Ann. Probab.} \textbf{31}(4):1917--1934, 2003.

\bibitem{SCZ}
L.~Saloff-Coste and T.~Zheng,
\newblock Random walks and isoperimetric profiles under moment conditions,
\newblock arXiv:1501.05929 

\bibitem{Tessera}
R.~Tessera,
\newblock Asymptotic isoperimetry on groups and uniform embeddings into Banach spaces,
\newblock \emph{Comment. Math. Helv.}, \textbf{86}(3):499--535, 2011.

\bibitem{Thomp}
R.~Thompson,
\newblock The rate of escape of random walks on polycyclic and metabelian groups,
\newblock \emph{Ann. Inst. Henri Poincar\'e Probab. Stat.} \textbf{49}(1):270--287, 2013. 

\bibitem{Var}
N.~T.~Varopoulos,
\newblock Long range estimates for Markov chains, 
\newblock \emph{Bull. Sci. Math. (2)} \textbf{109}(3):225--252, 1985.

\bibitem{Val}
A.~Valette,
\newblock Nouvelles approches de la propri\'et\'e (T) de Kazhdan,
\newblock \emph{Ast{\'e}risque} \textbf{294}(vii):97--124, 2004.

\bibitem{Woess}
W.~Woess,
\newblock \emph{Random Walks on Infinite Graphs and Groups}, Cambridge tracts in mathematics, \textbf{138}. 
\newblock Cambridge University Press, 2000. 


\end{thebibliography}
\end{document}